\documentclass{article}
\usepackage{amssymb}
\usepackage{latexsym}
\usepackage{amsmath}
\usepackage{amsthm,pstricks}
\topskip  \theoremstyle{remark}
\theoremstyle{plain}
\newtheorem{theorem}{Theorem}[section]
\newtheorem{lemma}[theorem]{Lemma}

\newtheorem{proposition}[theorem]{Proposition}

\newtheorem{example}{Example}[section]
\newtheorem{corollary}[theorem]{Corollary}

\begin{document}
\title{Counting $k$-ary words by number of adjacency differences of a prescribed size}

\author{Sela Fried\\
Department of Computer Science\\
Israel Academic College\\
52275 Ramat Gan, Israel\\
{\tt friedsela@gmail.com} \and
Toufik Mansour\\
Mathematics Department\\
University of Haifa\\
3498838 Haifa, Israel \\
{\tt tmansour@univ.haifa.ac.il} \and
Mark Shattuck\\
Mathematics Department\\
University of Tennessee\\
Knoxville, TN 37996\\
{\tt mark.shattuck2@gmail.com} }
\date{\small }
\maketitle

\begin{abstract}
Recently, the general problem of enumerating permutations $\pi=\pi_1\cdots \pi_n$  such that $\pi_{i+r}-\pi_i \neq s$ for all $1\leq i\leq n-r$, where $r$ and $s$ are fixed, was considered by Spahn and Zeilberger. In this paper, we consider an analogous problem on $k$-ary words involving the distribution of the corresponding statistic.  Note that for $k$-ary words, it suffices to consider only the $r=1$ case of the aforementioned problem on permutations.  Here, we compute for arbitrary $s$ an explicit formula for the ordinary generating function for $n \geq 0$ of the distribution of the statistic on $k$-ary words $\rho=\rho_1\cdots\rho_n$ recording the number of indices $i$ such that $\rho_{i+1}-\rho_i=s$. This result may then be used to find a comparable formula for finite set partitions with a fixed number of blocks, represented sequentially as restricted growth functions.  Further, several sequences from the OEIS arise as enumerators of certain classes of $k$-ary words avoiding adjacencies with a prescribed difference.  The comparable problem where one tracks indices $i$ such that the absolute difference $|a_{i+1}-a_i|$ is a fixed number is also considered on $k$-ary words and the corresponding generating function may be expressed in terms of Chebyshev polynomials.  Finally, combinatorial proofs are found for several related recurrences and formulas for the total number of adjacencies of the form $a(a+s)$ on the various structures.
\end{abstract}


\section{Introduction}

Spahn and Zeilberger \cite{SZ} considered the problem of counting permutations $\pi=\pi_1\cdots \pi_n$, represented in the one-line notation, that avoid certain adjacencies and raised the general questions of finding
$$u_{r,s}(n)=\#\{\pi \in \mathcal{S}_n: \pi_{i+r}-\pi_i \neq s \text{ for all } 1 \leq i \leq n-r\}$$
and
$$v_{r,s}(n)=\#\{\pi \in \mathcal{S}_n: |\pi_{i+r}-\pi_i| \neq s \text{ for all } 1 \leq i \leq n-r\},$$
for any $r,s \geq 1$.  They sought efficient methods for computing from scratch as many terms as possible from these sequences for any fixed $r$ and $s$.  To this end, they developed, using a weighted inclusion/exclusion type argument, general formulas for $u_{r,s}(n)$ and $v_{r,s}(n)$, which though they were complicated could be applied in some specific cases of $r$ and $s$.  They also supplied combinatorial proofs of the two-term and four-term recurrences satisfied by the sequences $u_{1,1}(n)$ and $v_{1,1}(n)$, respectively.

The problem of enumerating permutations that avoid adjacencies of certain letters is one that had arisen previously in various contexts.  For example, the sequence $v_{1,1}(n)$ has attracted considerable attention, since in addition to enumerating permutations of $[n]$ without rising or falling successions, it gives the number of ways of arranging $n$ non-attacking kings on an $n\times n$ board with one in each row and column. See, e.g., entry A002464 in the OEIS \cite{Sl} for further information.  Riordan \cite{Rior} derived a four-term recurrence for $v_{1,1}(n)$ and Robbins \cite{Rob} gave an explicit formula using an inclusion/exclusion argument.  An extension of the latter formula to $v_{r,r}(n)$ for any $r \geq 1$ was found by Tauraso \cite{Tau} in conjunction with his work on the rectangular case of the dinner table problem and was instrumental in establishing an asymptotic estimate for $v_{r,r}(n)$ for large $n$ and $r$ fixed.  Finally, Navarrete \cite{Nav} used the principle of inclusion/exclusion to derive the formula
$$u_{1,s}(n)=\sum_{j=0}^{n-s}(-1)^j\binom{n-s}{j}(n-j)!, \qquad n \geq s,$$
for each $s \geq 1$, which was subsequently shown in \cite{SZ}, using the Zeilberger algorithm \cite{Zeil}, to satisfy the two-term recurrence
$$u_{1,s}(n)=(n-1)u_{1,s}(n-1)+(n-s-1)u_{1,s}(n-2), \qquad n \geq s+1,$$
with initial values $u_{1,s}(n)=n!$ for $0 \leq n \leq s$.  A combinatorial proof was also given of this recurrence in the case $s=1$, and it is possible to extend this argument to general $s$.

By a $k$-\emph{ary word}, we mean a member of $[k]^n$ for some $n \geq0$, represented sequentially. Here, we consider an analogous problem on $k$-ary words to the one described above on permutations and study the distribution of the statistic on $[k]^n$ recording the number of occurrences of certain adjacencies (marked by a variable $q$).  When $q=0$, one obtains avoidance results paralleling those described above for permutations.  Let $\rho=\rho_1\cdots \rho_n \in [k]^n$.  When tracking the number of indices $i$ such that $\rho_{i+r}-\rho_i=s$ in members of $[k]^n$, it suffices to consider only the $r=1$ case (see remark at end of second section).  Let $\mu(\rho)$ and $\nu(\rho)$ denote the number of $i \in [n-1]$ such that $\rho_{i+1}-\rho_i=s$ and $|\rho_{i+1}-\rho_i|=s$, respectively.  We seek to find explicit formulas for the generating functions
$$A(x)=\sum_{n\geq0}\left(\sum_{\rho \in [k]^n}q^{\mu(\rho)}\right)x^n$$
and
$$B(x)=\sum_{n\geq0}\left(\sum_{\rho \in [k]^n}q^{\nu(\rho)}\right)x^n,$$
for arbitrary $s \geq 1$.  From the formulas for $A(x)$ and $B(x)$, one can derive recurrences for the respective distributions and letting $q=0$ gives results for the corresponding classes of $k$-ary words avoiding the adjacencies in question.  Other classes of $k$-ary words that have previously been studied which satisfy certain adjacency restrictions include smooth words \cite{MS3}, Catalan words \cite{BCR,BKV}, those avoiding a consecutive pattern \cite{BM} and others restricted by various statistics \cite{Ma}.

We also consider the problem of counting adjacencies with a given difference on the set of finite partitions of a prescribed size. By a \emph{partition} of $[n]$, we mean a collection of pairwise disjoint subsets, called \emph{blocks}, whose union is $[n]$.  Let $\mathcal{P}_{n,k}$ denote the set of partitions of $[n]$ with $k$ blocks and $\mathcal{P}_n=\cup_{k=0}^n\mathcal{P}_{n,k}$ the set of all partitions of $[n]$.  Recall that $|\mathcal{P}_{n,k}|=S(n,k)$, the Stirling number of the second kind, and $|\mathcal{P}_n|=B_n=\sum_{k=0}^nS(n,k)$, the $n$-th Bell number.  A partition $\Pi=T_1/T_2/\cdots$ is said to be in \emph{standard form} if its blocks $T_i$ satisfy $\text{min}(T_i)\leq \min(T_{i+1})$ for all $i \geq 1$.  In this case, then $\Pi$ may be represented sequentially by the canonical form $\pi=\pi_1\cdots\pi_n$ wherein $i \in B_{\pi_i}$ for $1 \leq i \leq n$ (see, e.g., \cite{Mi} or \cite{SW}). For example, if $\Pi=\{1,3,7\},\{2,4\},\{5,6\} \in \mathcal{P}_{7,3}$, then $\pi=1212331$.  Note that the canonical sequences $\pi$ satisfy the condition $\pi_{i+1} \leq \max(\pi_1\cdots\pi_i)+1$ for $1 \leq i \leq n-1$, with $\pi_1=1$, and thus are often referred to as \emph{restricted growth functions}.

In Section \ref{partitionsect}, we consider the statistic tracking the number of occurrences of $a(a+s)$ where $s \geq 2$ in the members $\pi \in \mathcal{P}_n$, represented sequentially as $\pi=\pi^{(1)}\pi^{(2)}\cdots$ where each $\pi^{(i)}$ is $i$-ary and starts with $i$.  We remark that this extends work begun in \cite{MS1}, where the $s=1$ case arose as a parameter tracking certain kinds of connectors between blocks of a finite partition. See also the related paper \cite{MM}.  For a variety of examples of statistics on the sequential forms of set partitions, we refer the reader to the text \cite{Man} and references contained therein.

The organization of this paper is as follows.  In the next section, we determine a formula for the generating function of the distribution on $k$-ary words for the parameter tracking the number of strings of the form $a(a+s)$ for a fixed $s\geq 1$.  We make use of linear algebra and Cramer's rule to derive our formula. In the third section, we consider several classes of $k$-ary words avoiding $a(a+s)$ corresponding to particular $k$ and $s$.  Among our results is a bijection between the set of $4$-ary words of length $n$ avoiding occurrences of 1-3 and 2-4 and the set of $(2^m-1)$-color compositions of $n+1$.  In the fourth section, we extend our generating function result when $s \geq 2$ to $\mathcal{P}_{n,k}$ for a fixed $k$.  From this, we may compute an explicit formula in terms of Stirling numbers for the total number of occurrences of $a(a+s)$ in all the members of $\mathcal{P}_{n,k}$.  Further, one may obtain the exponential generating function for the total number of occurrences in all the members of $\mathcal{P}_n$, from which it is possible to express the total on $\mathcal{P}_n$ in terms of Bell numbers.

In the fifth section, we consider the absolute difference and thus enumerate strings of the form $a(a\pm s)$ within $k$-ary words and compute an explicit formula for the generating function in terms of Chebyshev polynomials (see, e.g., \cite{Ri}). For this, we must differentiate two general cases wherein either $s+1 \leq k \leq 2s$ or $k \geq 2s+1$. In the latter apparently more difficult case, to prove our result, we make use of an LU-decomposition of the coefficient matrix for a (linear) system that arises involving certain unknown generating functions.  In the final section, we provide combinatorial proofs of several of the prior results which were shown by various algebraic methods, among them formulas for the total number of occurrences of $a(a+s)$ on the related classes of sequences as well as recurrences for the distributions of $a(a+s)$ and $a(a\pm s)$ on $k$-ary words.

\section{Distribution of $a(a+s)$ on $k$-ary words}

Let $\mathcal{A}_{n,k}$ denote the set of $k$-ary words of length $n$ and $\mathcal{A}_{n,k}^{(i)}$, the subset of $\mathcal{A}_{n,k}$ consisting of those words ending in $i$ for $1 \leq i \leq k$.   Given $k,s \geq 1$, let $a_n=a_n(q)=a_n^{(k,s)}(q)$ for $n \geq 1$ denote the distribution on $\mathcal{A}_{n,k}$ for the number of occurrences of the string $a(a+s)$ for some $a$, with $a_0=1$. To determine $a_n$, we refine it by considering the final letter and let $a_{n,i}=a_{n,i}(q)$ for $1 \leq i \leq k$ denote the restriction of the distribution $a_n$ to $\mathcal{A}_{n,k}^{(i)}$.

We have the following system of recurrences for the array $a_{n,i}$.

\begin{lemma}\label{lem1}
Let $k$ and $s$ be fixed positive integers, with $k \geq s+1$.  If $n \geq 2$, then
\begin{align}
a_{n,i}&=a_{n-1}, \qquad 1 \leq i \leq s, \label{aneq1}\\
a_{n,i}&=a_{n-1}+(q-1)a_{n-1,i-s}, \qquad s+1 \leq i \leq k, \label{aneq2}
\end{align}
with $a_{1,i}=1$ for $1 \leq i \leq k$.
\end{lemma}
\begin{proof}
We may assume $n \geq 2$, the initial condition being clear.  Note that one may append $i$ for each $i \in [s]$  to any member of $\mathcal{A}_{n-1,k}$ without introducing an occurrence of $a(a+s)$, which implies \eqref{aneq1}. On the other hand, if $i \in [s+1,k]$, then appending $i$ introduces an occurrence of $a(a+s)$ when added to a member of $\mathcal{A}_{n-1,k}^{(i-s)}$, with no such change occurring for members of $\mathcal{A}_{n-1,k}-\mathcal{A}_{n-1,k}^{(i-s)}$.  Combining these cases then implies $a_{n,i}=qa_{n-1,i-s}+a_{n-1}-a_{n-1,i-s}$, which gives \eqref{aneq2}.
\end{proof}

Let ${\bf C}_m=c_{i,j}$ denote the $m \times m$ matrix in which $c_{i,i+1}=1$ for all $i$ and $c_{i,i-s+1}=z$ for $s \leq i \leq m$, with $c_{i,j}=0$ for all other $i$, $j$.  We will need the following determinant formula.

\begin{lemma}\label{lem2}
Let $m \geq 1$ be given by $m=ds+r$, where $d \geq0$, $s \geq 1$ and $0 \leq r \leq s-1$.  Then we have
\begin{equation}\label{lem2e1}
\det({\bf C}_m)=\begin{cases}
       (-1)^{m-d}z^d, & \text{if $r=0$;} \\
         0, & \text{if $1 \leq r \leq s-1$.}\\
    \end{cases}
\end{equation}
\end{lemma}
\begin{proof}
If $s=1$, the result is apparent, so assume $s>1$.  The cases of \eqref{lem2e1} for $1 \leq m < s$ are readily verified, so assume $m\geq s$.  Recall the definition of the determinant
\begin{equation}\label{lem2e2}
\det({\bf C}_m)=\sum_{\sigma \in \mathcal{S}_m}(-1)^{\text{sgn}(\sigma)}c_{1,\sigma(1)}\cdots c_{m,\sigma(m)},
\end{equation}
where $\mathcal{S}_m$ is the set of permutations of $[m]$ and $\text{sgn}(\sigma)$ denotes the sign of the permutation $\sigma$.  Note that in order for the term in \eqref{lem2e2} corresponding to $\sigma$ to be non-zero, we must have $\sigma(1)=2,\sigma(2)=3,\ldots,\sigma(s-1)=s$.  This in turn implies $\sigma(s)=1$, for otherwise the factor $c_{\sigma^{-1}(1),1}$ would be zero in the product corresponding to $\sigma$ in \eqref{lem2e2}, as there is only one non-zero entry in the first column of ${\bf C}_m$.  Thus, $\sigma$ must contain the cycle $(1,2,\ldots,s)$. By similar reasoning, we then have that $\sigma$ must contain the cycles $(s+1,\ldots,2s)$, $(2s+1,\ldots,3s)$ and so on. If $m$ is not divisible by $s$, then not all of the cycles can have this form.  Indeed, if $\ell_0$ is the smallest $\ell$ such that $(\ell s+1,\ldots,(\ell+1)s)$ does not occur as a cycle of $\sigma$, then for some $x \in [\ell_0 s+1,(\ell_0+1)s]$, we must have $c_{x,\sigma(x)}=0$, which implies the second case of \eqref{lem2e1}.  If $m=ds$, then only $\sigma=(1,\ldots,s),\ldots,((d-1)s+1,\ldots,ds)$ can make a non-zero contribution to \eqref{lem2e2} and it is seen to be given by $(-1)^{m-d}z^d$, which implies the first case of \eqref{lem2e1} and completes the proof.
\end{proof}

Define the generating functions $A(x)=\sum_{n\geq0}a_nx^n$ and $A_i(x)=\sum_{n\geq1}a_{n,i}x^n$. We have the following explicit formula for $A(x)$.

\begin{theorem}\label{th1}
Let $k=k'+sk''$, with $1\leq k'\leq s$ and $k''\geq0$. Then the generating function for the distribution of the number of occurrences of $a(a+s)$ on the set of $k$-ary words of length $n$ for $n \geq 0$ is given by
\begin{equation}\label{th1e0}
A(x)=
\frac{(1+(1-q)x)^2}
{1-(k-2+2q)x-(k+s-1+q)(1-q)x^2+(k'(1+(1-q)x)-s)x((q-1)x)^{k''+1}}.
\end{equation}
\end{theorem}
\begin{proof}
Note first that $A(x)$ reduces to $\frac{1}{1-kx}$, as it should, when $k''=0$, and hence we may assume $k''\geq 1$, i.e., $k \geq s+1$.  By \eqref{aneq1} and \eqref{aneq2}, we have $A_i(x)=xA(x)$ for $1 \leq i \leq s$ and $A_i(x)+(1-q)xA_{i-s}(x)=xA(x)$ for $s+1 \leq i \leq k$.

Let ${\bf A}_k=(d_{i,j})$ be the $k\times k$ matrix such that $d_{i,i}=1$, $d_{i,i-s}=(1-q)x$ and $d_{i,j}=0$ for all other $i,j$. Let ${\bf B}_{k;i}$ for $1 \leq i \leq k$ denote the matrix obtained from ${\bf A}_k$ by replacing its $i$-th column with $(1,1,\ldots,1)^t$.  By Cramer's rule, we have
$$A_i(x)=\frac{xA(x)\det({\bf B}_{k;i})}{\det({\bf A}_k)}.$$
Expanding repeatedly along the final column implies $\det({\bf B}_{k;i})=\det({\bf B}_{i;i})$. Since clearly $\det({\bf A}_k)=1$, we have
\begin{align}
A_i(x)=xA(x)\det({\bf B}_{i;i}).\label{eqt0}
\end{align}
Let $b_i=\det({\bf B}_{i,i})$ and $c_i=\det({\bf C}_i)$, where ${\bf C}_i$ is the matrix from Lemma \ref{lem2} with $z=(1-q)x$. Expanding along the first row, we obtain
\begin{equation}\label{th1e1}
b_i=b_{i-1}+(-1)^{i-1}c_{i-1}, \qquad i \geq 2,
\end{equation}
with $b_1=1$. Iterating \eqref{th1e1}, and making use of Lemma \ref{lem2}, yields
\begin{equation}\label{detB}
\det({\bf B}_{ds+r;ds+r})=b_{ds+r}=\sum_{i=0}^d(-(1-q)x)^i=\frac{1-(-(1-q)x)^{d+1}}{1+(1-q)x},
\end{equation}
where here $d \geq 0$ and $1\leq r\leq s$. Hence, by \eqref{eqt0} and the fact  $$A(x)=1+\sum_{i=1}^kA_i(x)=1+\sum_{r=1}^{k'}\sum_{d=0}^{k''}A_{ds+r}(x)+\sum_{r=k'+1}^s\sum_{d=0}^{k''-1}A_{ds+r}(x),$$ we obtain
\begin{align*}
A(x)&=1+xA(x)\left(\sum_{r=1}^{k'}\sum_{d=0}^{k''}\det({\bf B}_{ds+r;ds+r})+\sum_{r=k'+1}^s\sum_{d=0}^{k''-1}\det({\bf B}_{ds+r;ds+r})\right)\\
&=1+xA(x)\left(\sum_{r=1}^{k'}\sum_{d=0}^{k''}\frac{1-((q-1)x)^{d+1}}{1+(1-q)x}\sum_{r=k'+1}^s\sum_{d=0}^{k''-1}\frac{1-((q-1)x)^{d+1}}{1+(1-q)x}\right)\\
&=1+\frac{k'xA(x)}{1+(1-q)x}\left(k''+1-\frac{(q-1)x-((q-1)x)^{k''+2}}{1+(1-q)x}\right)\\
&\quad+\frac{(s-k')xA(x)}{1+(1-q)x}\left(k''-\frac{(q-1)x-((q-1)x)^{k''+1}}{1+(1-q)x}\right),
\end{align*}
which is equivalent to
$$A(x)=1+xA(x)
\frac{k'+sk''+(k'+sk''+s)(1-q)x-(k'(1+(1-q)x)-s)((q-1)x)^{k''+1}}{(1+(1-q)x)^2}.$$
Solving for $A(x)$ leads to the result.
\end{proof}

Note that $A(x)$ reduces to $\frac{1}{1-kx}$ when one takes $q=1$ in \eqref{th1e0}, as required.  Letting $q=0$ in \eqref{th1e0} yields the following result.

\begin{corollary}\label{cor1}
The generating function for the number of members of $\mathcal{A}_{n,k}$ having no occurrences of the string $a(a+s)$ is given by
\begin{equation}\label{cor1e1}
A(x)\mid_{q=0}=
\frac{(1+x)^2}
{1-(k-2)x-(k+s-1)x^2-(k'-s+k'x)(-x)^{k''+2}}.
\end{equation}
In particular, we have the recurrence
$$a_n(0)=(k-2)a_{n-1}(0)+(k+s-1)a_{n-2}(0)+(-1)^{k''}(k'-s)a_{n-k''-2}(0)+(-1)^{k''}k'a_{n-k''-3}(0),$$
for all $n \geq k''+3$.
\end{corollary}

We consider in greater detail some specific cases when $q=0$ in the subsequent section.

Note that the denominator in formula \eqref{th1e0} has a factor of $(1+(1-q)x)^2$.  Cancelling out this common factor with the numerator, and synthetically dividing $1+(1-q)x$ twice into the denominator, yields the alternative formula
\begin{equation}\label{th1alteq}
A(x)=\frac{1}{1-s_{k''}},
\end{equation}
where
$$s_{k''}=\sum_{i=0}^{k''}(q-1)^i(k-is)x^{i+1}.$$
Thus, we have the following alternative recurrence for $a_n$ when $n \geq k''+1$:
\begin{equation}\label{analtrec}
a_n=ka_{n-1}+(q-1)(k-s)a_{n-2}+(q-1)^2(k-2s)a_{n-3}+\cdots+(q-1)^{k''}(k-k''s)a_{n-k''-1},
\end{equation}
and, in particular,
\begin{equation}\label{analtrecq=0}
a_n(0)=ka_{n-1}(0)-(k-s)a_{n-2}(0)+(k-2s)a_{n-3}(0)+\cdots+(-1)^{k''}(k-k''s)a_{n-k''-1}(0).
\end{equation}

Differentiating both sides of {\eqref{th1e0} with respect to $q$, and setting $q=1$, gives
\begin{align*}
\frac{\partial}{\partial q}A(x)\mid_{q=1}&=\frac{x}{(1-kx)^2}\left(2-(k+s)x-\delta_{k'',0}(k'-s)x\right)-\frac{2x}{1-kx}\\
&=\frac{(k-s)x^2-\delta_{k'',0}(k'-s)x^2}{(1-kx)^2}=\begin{cases}
       0, & \text{if $k''=0$;} \\
         \frac{(k-s)x^2}{(1-kx)^2}, & \text{if $k''>0$.}
    \end{cases}
\end{align*}
Extracting the coefficient of $x^n$ yields the following result.

\begin{corollary}\label{cor2}
Let $n \geq 1$ and $k \geq s+1$. Then the total number of occurrences of the string $a(a+s)$ in all the members of $\mathcal{A}_{n,k}$ is given by $k^{n-2}(k-s)(n-1)$.
\end{corollary}

The preceding formula may also be realized combinatorially as follows.  Equivalently, we count marked members of $\mathcal{A}_{n,k}$ wherein an occurrence of $a(a+s)$ is marked.  Suppose $\pi=\pi_1\cdots \pi_n \in \mathcal{A}_{n,k}$ such that $\pi_i\pi_{i+1}$ for some $i$ corresponds to the marked occurrence of $a(a+s)$.  Then there are $n-1$ choices for the index $i$, and once this is chosen, $k^{n-2}$ possibilities for the letters of $\pi$ occurring in positions other than the $i$-th and $(i+1)$-st. Finally, there are $k-s$ choices for $\pi_i$, since it must belong to $[k-s]$, which determines $\pi_{i+1}$ and implies that there are $k^{n-2}(k-s)(n-1)$ marked members of $\mathcal{A}_{n,k}$.

\emph{Remark:} Given $r \geq 1$, let $j_n=j_n^{(r)}(q)$ denote the distribution for the statistic on $\mathcal{A}_{n,k}$ recording the number of indices $i \in [n-r]$ such that $\pi_{i+r}-\pi_i=s$ within  $\pi=\pi_1\cdots\pi_n$.  Finding $j_n$ then reduces to the $r=1$ case discussed above.  To see this, let $n=dr+t$, where $d \geq0$ and $0 \leq t \leq r-1$.  Consider the contributions to $j_n$ coming from each of the $r$ subsequences $S_i$ of $\pi$ given by
\begin{align*}
S_1&=\pi_1\pi_{r+1}\cdots\pi_{dr+1},\ldots,S_t=\pi_t\pi_{r+t}\cdots \pi_{dr+t},\\
S_{t+1}&=\pi_{t+1}\pi_{r+t+1}\cdots\pi_{(d-1)r+t+1},\ldots,S_{r-1}=\pi_{r-1}\pi_{2r-1}\cdots\pi_{(d-1)r+r-1},\\
S_r&=\pi_r\pi_{2r}\cdots\pi_{dr}.
\end{align*}
Since the letters in these subsequences may be chosen independently without affecting the corresponding distributions on each one, we have $j_n=a_{d+1}^ta_d^{r-t}$, where $a_n$ can be determined using \eqref{th1e0}.  A similar remark applies when one considers as in Section \ref{apm s} below the number of indices $i$ such that the absolute difference $|\pi_{i+r}-\pi_i|$ is a prescribed number.

\section{Some specific avoidance cases}

In this section, we consider several specific cases of $k$ and $s$ in $a_n^{(k,s)}(0)$.  First note that when $s=1$, one gets entry A277666 in the OEIS \cite{Sl}, which gives the anti-diagonals (in the natural order) of the square array $w_{n,k}$ for the number of members of $\mathcal{A}_{n,k}$ in which the string $a(a+1)$ fails to occur for any $a$, where one assumes $w_{n,0}=\delta_{n,0}$ and $w_{0,k}=1$ for all $n$ and $k$.

Let $F_n=F_{n-1}+F_{n-2}$ for $n \geq 2$ denote the $n$-th Fibonacci number, with $F_0=0$, $F_1=1$. Taking $k=3$ and $s=2$ in \eqref{th1alteq} at $q=0$ gives
$$A(x)|_{k=3,s=2,q=0}=\frac{1}{1-3x+x^2}=\sum_{n\geq 0}F_{2n+2}x^n,$$
and hence $a_n^{(3,2)}(0)=F_{2n+2}$ for all $n \geq 0$.

Taking $k=4$ and $s=2$ in \eqref{th1alteq} gives
$$A(x)\mid_{k=4,s=2,q=0}=\frac{1}{1-4x+2x^2}.$$
Comparing generating functions, this implies $a_n^{(4,2)}(0)=A007070(n)$ for all $n \geq 0$.

Taking $k=5$ and $s=2$ in \eqref{th1alteq} gives
\begin{align*}
A(x)\mid_{k=5,s=2,q=0}&=\frac{1}{1-5x+3x^2-x^3}=\frac{1}{x^3}\left(\frac{1-5x+3x^2}{1-5x+3x^2-x^3}-1\right)\\
&=\sum_{n\geq3}A200676(n)x^{n-3}=\sum_{n\geq 0}A200676(n+3)x^n,
\end{align*}
and hence $a_n^{(5,2)}(0)=A200676(n+3)$ for all $n \geq 0$.

To summarize, we have the following result.

\begin{corollary}\label{cor3}
If $n \geq 0$, then $a_n^{(3,2)}(0)=F_{2n+2}$, $a_n^{(4,2)}(0)=A007070(n)$ and $a_n^{(5,2)}(0)=A200676(n+3)$.
\end{corollary}

It is also possible to explain the prior equalities directly without recourse to generating functions, see final section.

By a $(2^m-1)$-\emph{color} composition, we mean one in which a part of size $i$ for each $i$ can come in one of $2^i-1$ colors.  (For an introduction to $m$-color compositions, we refer the reader to the seminal paper \cite{Ag}; see also \cite{E,MS2}, where functions of $m$ are considered.) Among the combinatorial interpretations for A007070$(n)$ is that it enumerates the set of $(2^m-1)$-color compositions of $n+1$ for $n \geq0$, which raises the question of finding a bijection with the class of words enumerated by $a_n^{(4,2)}(0)$.

\begin{proposition}\label{prop1}
There exists for each $n \geq 0$ a bijection between the set of $(2^m-1)$-color compositions of $n+1$ and the set of $4$-ary words of length $n$ in which the strings 1-3 and 2-4 do not occur.
\end{proposition}
\begin{proof}
Let $\mathcal{U}_n$ and $\mathcal{W}_n$ denote the respective sets in question.  First note that a member of $\mathcal{U}_n$ may be represented geometrically as a sequence of dots and vertical bars (which separate two adjacent dots), where there are $n+1$ dots in all any one of which may be circled such that at least one circled dot must occur between any two consecutive  bars as well as prior to the first and after the last bar.  The number of dots $i$ between two consecutive bars may be viewed as a part of size $i$ for some $i \geq 1$ which receives one of $2^i-1$ possible colors based on the non-empty subset of the dots occurring between the bars that are circled (for which there are $2^i-1$ possibilities).  Putting together the sequence of (colored) parts, including those corresponding to dots occurring prior to the first and after the last bars, yields a $(2^{m}-1)$-color composition of $n+1$.  Note that one-part compositions correspond to the case in which there are no vertical bars wherein some non-empty subset of the dots is circled.

We may now encode members of $\mathcal{U}_n$, represented geometrically as described above, as follows. Starting with the single member of $\mathcal{U}_0$, represented as a circled dot, we perform a sequence of $n$ operations, each being one of the four kinds: (1) add a vertical bar after the final dot, followed by a circled dot, (2) add a circled dot to the end, (3) add an uncircled dot to the end or (4) insert an uncircled dot directly prior to the terminal circled dot.  Further, we require that operation 4 be performed only directly after a step where 1 or 4 was performed (or as the first step in the sequence).  Note that, in the $j$-th step of the procedure for each $1 \leq j \leq n$, one of the four maneuvers is applied to a member of $\mathcal{U}_{j-1}$ and results in a member of $\mathcal{U}_{j}$. Further, it is seen that all of the distinct geometric representations of members of $\mathcal{U}_n$ arise uniquely by performing various sequences of $n$ operations, where in each step, one of the maneuvers 1 through 4 is executed such that a 4 never directly follows a 2 or 3.

Let $\mathcal{V}_n$ denote the set of $4$-ary sequences of length $n$ in which the strings 2-4 and 3-4 do not occur.  It then suffices to define a bijection between the sets $\mathcal{V}_n$ and $\mathcal{W}_n$. Let $\pi=\pi_1\cdots \pi_n \in \mathcal{V}_n$.   Consider replacing each string within $\pi$ of the form $1^d3$, where $d$ is maximal, with $34^d$.  Let $\pi'$ denote the resulting $4$-ary word.  Note that $\pi$ avoiding 3-4 implies the operation $\pi \mapsto \pi'$ is reversible, and the maximality of $d$ in each instance of the replacement implies $\pi'$ avoids 1-3.  Further, it is seen that no occurrences of 2-4 are introduced by these replacements.  Hence, $\pi \mapsto \pi'$ provides a bijection between $\mathcal{V}_n$ and $\mathcal{W}_n$, which when composed with the bijection between $\mathcal{U}_n$ and $\mathcal{V}_n$ implicit with the encoding described above, yields the desired bijection between $\mathcal{U}_n$ and $\mathcal{W}_n$.
\end{proof}

\subsection{A polynomial generalization of $F_{2n}$}

In this subsection, we consider a bivariate polynomial generalization of the Fibonacci numbers of even index by considering a pair of statistics on the set $\mathcal{J}_n$ of $3$-ary words of length $n$ in which 1-3 does not occur. We remark that it is possible to obtain polynomial generalizations of other similar sequences (such as those mentioned above from \cite{Sl}) for a fixed $k$ and $s$ by an analysis comparable to the one below.  By a \emph{level} and an \emph{ascent} within a sequence $\pi=\pi_1\cdots \pi_n$, we mean an index $i \in [n-1]$ such that $\pi_i=\pi_{i+1}$ and $\pi_i<\pi_{i+1}$, respectively. Let $\text{lev}(\pi)$ and $\text{asc}(\pi)$ denote the number of levels and ascents of $\pi$, respectively.

Note that the set equivalent to $\mathcal{J}_n$ consisting of the ternary sequences of length $n$ with letters in $\{0,1,2\}$ such that $1$ never immediately follows $0$ is mentioned in \cite[p.\,46,\,Exercise\,14(i)]{Stan}, where the problem of its enumeration is raised.  See also \cite[p.\,150]{BQ}, where a bijection with the set of square-and-domino tilings of length $2n+1$ is described.

Consider the following joint distribution $j_n=j_n(p,q)$ defined by
$$j_n(p,q)=\sum_{\pi\in \mathcal{J}_n}p^{\text{lev}(\pi)}q^{\text{asc}(\pi)}, \qquad n \geq 1.$$
Note that $j_n(1,1)=|\mathcal{J}_n|=F_{2n+2}$ for all $n$. Let $j_{n,i}$ for $1 \leq i \leq 3$ denote the restriction of $j_n$ to the subset of $\mathcal{J}_n$ whose members end in $i$. Considering the last two letters within a member of $\mathcal{J}_n$ implies the following recurrences for $n \geq 2$:
\begin{align*}
j_{n,1}&=pj_{n-1,1}+j_{n-1,2}+j_{n-1,3},\\
j_{n,2}&=qj_{n-1,1}+pj_{n-1,2}+j_{n-1,3},\\
j_{n,3}&=qj_{n-1,2}+pj_{n-1,3},
\end{align*}
with $j_{1,i}=1$ for each $i$.

Let $f_i=f_i(x)$ for $1 \leq i \leq 3$ be defined by $f_i=\sum_{n\geq 1}j_{n,i}x^n$.  Then the preceding recurrences may be rewritten in terms of the $f_i$ as
\begin{align*}
(1-px)f_1&=x+x(f_2+f_3),\\
(1-px)f_2&=x+x(qf_1+f_3),\\
(1-px)f_3&=x+qxf_2.
\end{align*}
Solving the preceding system for $f_1$, $f_2$ and $f_3$ yields
\begin{align*}
f_1&=\frac{x(1+(1-p)x)^2}{1-3px+(3p^2-2q)x^2+(2pq-p^3-q^2)x^3},\\
f_2&=\frac{x(1+(1+q-2p)x+(p-1)(p-q)x^2)}{1-3px+(3p^2-2q)x^2+(2pq-p^3-q^2)x^3},\\
f_3&=\frac{x(1+(q-2p)x+(p^2+q^2-pq-q)x^2)}{1-3px+(3p^2-2q)x^2+(2pq-p^3-q^2)x^3}.
\end{align*}

Let $f(x)=f(x;p,q)$ be given by $f(x)=1+\sum_{n\geq 1}j_nx^n=1+f_1(x)+f_2(x)+f_3(x)$.  We then have the following explicit formula for $f(x)$.

\begin{theorem}\label{Fnpoly}
The generating function for the joint distribution of the level and ascent statistics on $\mathcal{J}_n$ for $n \geq 0$ is given by
$$\frac{(1+(1-p)x)^3}{1-3px+(3p^2-2q)x^2+(2pq-p^3-q^2)x^3}.$$
\end{theorem}

Taking $p=q=1$ in Theorem \ref{Fnpoly} gives $f(x;1,1)=\frac{1}{1-3x+x^2}=\sum_{n\geq 0}F_{2n+2}x^n$, as required.

Setting $p$ or $q$ equal to zero or differentiation with respect to $p$ or $q$ in Theorem \ref{Fnpoly} yields respectively the following results.

\begin{corollary}\label{Fnpolycor1}
There are $F_{n+3}$ members of $\mathcal{J}_n$ for $n \geq 1$ that have no levels and $\binom{n+2}{2}$ that have no ascents.
\end{corollary}

\begin{corollary}\label{Fnpolycor2}
There are altogether $(n-1)F_{2n}$ levels and $\frac{2(n-1)L_{2n}-4F_{2n-2}}{5}$ ascents in all the members of $\mathcal{J}_n$ for $n \geq 1$, where $L_n=F_{n+1}+F_{n-1}$ denotes the $n$-th Lucas number.
\end{corollary}

Finally, let $\text{des}(\pi)$ denote the number of \emph{descents}, i.e., indices $i \in [n-1]$ such that $\pi_i>\pi_{i+1}$ in $\pi=\pi_1\cdots \pi_n$.  Since $\text{lev}(\pi)+\text{asc}(\pi)+\text{des}(\pi)=n-1$ for all $\pi$ of length $n$, we have that the generating function for the joint distribution of the level and descent statistics on $\mathcal{J}_n$ for $n \geq 1$ is given by
\begin{equation}\label{deseq}
(1/q)(f(qx;p/q,1/q)-1)=\frac{3x+(3q-6p+2)x^2+(3p^2+q^2-3pq-2p+1)x^3}{1-3px+(3p^2-2q)x^2+(2pq-p^3-q)x^3}.
\end{equation}
Taking $q=0$ in \eqref{deseq} implies that there are $\binom{n+2}{2}-n+1$ members of $\mathcal{J}_n$ that have no descents, i.e., are weakly increasing, and differentiation with respect to $q$ implies that the total number of descents in all the members of $\mathcal{J}_n$ is given by $\frac{(n-1)L_{2n+1}+4F_{2n-2}}{5}$.  These formulas may also be explained combinatorially by extending the arguments given for Corollaries \ref{Fnpolycor1} and \ref{Fnpolycor2} in the final section.

\emph{Remark:} Let $\text{tot}(x)$ denote the total number of occurrences of a subword $x$ over all members of $\mathcal{J}_n$, where $x$ is one of either level, ascent or descent.  Then, from the formulas found above, one can show
$$\text{tot}(\text{lev})>\text{tot}(\text{asc})>\text{tot}(\text{des}), \qquad n \geq 5.$$
That $\text{tot}(\text{lev})>\text{tot}(\text{asc})$ follows intuitively from the fact that any letter in $\{1,2,3\}$ can start a level, whereas only 1 or 2 can start an ascent (with only one option for the ascent top in either case as 1-3 is forbidden).  Further, $\text{tot}(\text{asc})>\text{tot}(\text{des})$ is seen to hold since 1 is more likely the second letter of a descent in $\mathcal{J}_n$ than is 2, with 1 subsequently only having two choices for its successor rather than three for a 2.  On the other hand, both possibilities for the second letter in an ascent, namely 2 or 3, have three choices for their successor.

\section{Distribution of $a(a+s)$ on set partitions}\label{partitionsect}

We consider in this section the statistic on $\mathcal{P}_n$ tracking the number of occurrences of $a(a+s)$, where $s \geq 2$ and members of $\mathcal{P}_n$ are identified in terms of their sequential forms.  Our first result is a formula for the generating function of the distribution of this statistic on $\mathcal{P}_{n,k}$ for a fixed $k$, which we will denote by $P_k=P_k(x;q)$. Note that we may restrict attention to the case $k \geq s+1$ since clearly no member of $\mathcal{P}_{n,k}$ where $k \in [s]$ can contain an occurrence of $a(a+s)$.

\begin{theorem}\label{gfpartit}
Let $k \geq s+1$ and $s \geq 2$ be fixed, with $k=k'+k''s$ where $1 \leq k' \leq s$ and $k''\geq 1$.  Then the generating function of the distribution on $\mathcal{P}_{n,k}$ for $n \geq k$ for the number of occurrences of the string $a(a+s)$ is given by
\begin{align}
P_k(x;q)&=\frac{x^k(1+(q-1)x)(1+(1-q))x)^{k-s+1}\left(1-\left((q-1)x\right)^{k''+1}\right)^{k'-1}}{\prod_{j=1}^s(1-jx)}\notag\\
 &\quad \times \prod_{\ell=1}^{k''-1}\left(1-\left((q-1)x\right)^{\ell+1}\right)^{s-1}\left(1-\left((q-1)x\right)^{\ell+2}\right)\times \prod_{j=s+1}^k D_j(x;q),\label{gfpartite1}
 \end{align}
 where $D_j(x;q)$ for $j \geq s+1$ is given by
$$  D_j(x;q)=\frac{1}{1-(j-2+2q)x-(j+s-1+q)(1-q)x^2+(j'(1+(1-q)x)-s)x\left((q-1)x\right)^{j''+1}},$$
 with $j=j'+j''s$ where $1 \leq j' \leq s$ and $j''\geq 1$.
 \end{theorem}
\begin{proof}
Let $\pi \in \mathcal{P}_{n,k}$ be represented as $\pi=\pi^{(1)}\cdots \pi^{(k)}$, where each $\pi^{(i)}$ is $i$-ary and starts with $i$.  Note first that for each $i \in [s-1]$, the contribution of $\pi^{(i)}$ towards $P_k$ is given by $\frac{x}{1-ix}$.  If $i=s$, then words $\pi^{(s)}$ ending in $1$ contribute an extra factor of $q$ due to $\pi^{(s)}$ being directly followed by an $s+1$.  Hence, by subtraction, we have that $\pi^{(s)}$ contributes a factor of $\frac{x+(q-1)x^2}{1-sx}$.  Define
$$E_j(x)=\frac{(1+(1-q)x)^2}
{1-(j-2+2q)x-(j+s-1+q)(1-q)x^2+(j'(1+(1-q)x)-s)x((q-1)x)^{j''+1}}$$
and
$$E_j'(x)=E_j(x)\left(1+(q-1)x\det({\bf B}_{j-s+1;j-s+1})\right),$$
for $j\geq s+1$, where $j=j'+j''s$ is as described.  By \eqref{th1e0} and \eqref{eqt0}, it is seen that $xE_j'(x)$ gives the contribution towards $P_k$ of the section $\pi^{(j)}$ for $s<j<k$, with $\pi^{(k)}$ contributing $xE_k(x)$.  To see this, let $\pi^{(j)}=j\rho$ where $s<j<k$, and we must subtract the possibility of $\rho$ ending in $j-s+1$, which has generating function $xE_j(x)\det({\bf{B}}_{j-s+1;j-s+1})$.  This term must then be added back with an extra factor of $q$ to account for the occurrence of $a(a+s)$ whose first letter is the last of $\rho$.    Combining the various cases above, we have
\begin{equation}\label{gfpartite2}
P_k(x;q)=\frac{x^k(1+(q-1)x)E_k(x)}{\prod_{j=1}^s(1-jx)}\times \prod_{\ell=1}^{k''-1}\prod_{i=1}^sE'_{\ell s+i}(x)\times \prod_{r=1}^{k'-1}E'_{k''s+r}(x).
\end{equation}

Now observe that if $1 \leq \ell \leq k''-1$ and $1 \leq i \leq s-1$, then
\begin{align*}
&1+(q-1)x\det({\bf{B}}_{(\ell-1)s+i+1;(\ell-1)s+i+1})=1+(q-1)x\left(\frac{1-((q-1)x)^{\ell}}{1-(q-1)x}\right)=\frac{1-((q-1)x)^{\ell+1}}{1-(q-1)x},
\end{align*}
by \eqref{detB}. If $1 \leq \ell \leq k''-1$ and $i=s$, then we get
$$1+(q-1)x\det({\bf{B}}_{\ell s+1;\ell s+1})=1+(q-1)x\left(\frac{1-((q-1)x)^{\ell+1}}{1-(q-1)x}\right)=\frac{1-((q-1)x)^{\ell+2}}{1-(q-1)x}.$$
Thus, we have
\small\begin{align*}
&\prod_{\ell=1}^{k''-1}\prod_{i=1}^sE'_{\ell s+i}(x)=\prod_{\ell=1}^{k''-1}\left(1-((q-1)x)^{\ell+1}\right)^{s-1}\left(1-((q-1)x)^{\ell+2}\right)(1+(1-q)x)^s\\
        &\quad\times\prod_{\ell=1}^{k''-1}\prod_{i=1}^s\frac{1}{1-(\ell s+i-2+2q)x-((\ell+1)s+i-1+q)(1-q)x^2+(i(1+(1-q)x)-s)x\left((q-1)x\right)^{\ell+1}}\\
        &=(1+(1-q)x)^{k-k'-s}\prod_{\ell=1}^{k''-1}\left(1-((q-1)x)^{\ell+1}\right)^{s-1}\left(1-((q-1)x)^{\ell+2}\right)\times\prod_{j=s+1}^{k-k'}D_j(x;q).
\end{align*}
\normalsize By similar reasoning, we have
\small\begin{align*}
&\prod_{r=1}^{k'-1}E'_{k''s+r}(x)\\
&=\prod_{r=1}^{k'-1}\frac{(1+(1-q)x)\left(1-((q-1)x)^{k''+1}\right)}{1-(k''s+r-2+2q)x-((k''+1)s+r-1+q)(1-q)x^2+(r(1+(1-q)x)-s)x\left((q-1)x\right)^{k''+1}}\\
&=(1+(1-q)x)^{k'-1}\left(1-((q-1)x)^{k''+1}\right)^{k'-1}\prod_{j=k-k'+1}^{k-1}D_j(x;q),
\end{align*}
\normalsize with $E_k(x)=(1+(1-q)x)^2D_k(x;q)$.
Formula \eqref{gfpartite1} now follows from \eqref{gfpartite2} and combining the expressions found above for the various components of the product.\end{proof}

Note that \eqref{gfpartite1} recovers the well-known generating function formula $\sum_{n\geq k}S(n,k)x^n=\frac{x^k}{\prod_{j=1}^k(1-jx)}$ when $q=1$. Differentiation with respect to $q$ yields the following result.

\begin{corollary}\label{totPnk}
The total number of occurrences of the form $a(a+s)$ in all the members of $\mathcal{P}_{n,k}$, where $k \geq s+1$ and $s \geq 2$, is given by
$$(k-s)S(n-1,k)+\sum_{j=s+1}^k\sum_{t=0}^{n-k-2}S(n-t-2,k)j^t(j-s), \qquad n \geq k+1.$$
\end{corollary}
\begin{proof}
Differentiating both sides of \eqref{gfpartite1} with respect to $q$, and setting $q=1$, gives
\begin{equation}\label{totogf}
\frac{\partial}{\partial q}P_k(x;q)\mid_{q=1}=\frac{x^k}{\prod_{i=1}^k(1-ix)}\left(-(k-s)x+\sum_{j=s+1}^k\frac{2x-(j+s)x^2}{1-jx}\right).
\end{equation}
Extracting the coefficient of $x^n$, and manipulating the resulting summation formula, implies
\begin{align*}
&[x^n]\frac{\partial}{\partial q}P(x;q)\mid_{q=1}\\
&=-(k-s)S(n-1,k)+2\sum_{j=s+1}^k\sum_{t=0}^{n-k-1}S(n-t-1,k)j^t-\sum_{j=s+1}^k\sum_{t=0}^{n-k-2}S(n-t-2,k)j^t(j+s)\\
&=(k-s)S(n-1,k)+2\sum_{j=s+1}^k\sum_{t=1}^{n-k-1}S(n-t-1,k)j^t-\sum_{j=s+1}^k\sum_{t=1}^{n-k-1}S(n-t-1,k)j^t\\
&\quad-s\sum_{j=s+1}^k\sum_{t=0}^{n-k-2}S(n-t-2,k)j^t\\
&=(k-s)S(n-1,k)+\sum_{j=s+1}^k\sum_{t=0}^{n-k-2}S(n-t-2,k)j^{t+1}-s\sum_{j=s+1}^k\sum_{t=0}^{n-k-2}S(n-t-2,k)j^t\\
&=(k-s)S(n-1,k)+\sum_{j=s+1}^k\sum_{t=0}^{n-k-2}S(n-t-2,k)j^t(j-s),
\end{align*}
as desired.
\end{proof}

\emph{Remark:} Note that formulas \eqref{gfpartite1} and \eqref{gfpartite2} do not hold in the case $s=1$.  This is due in part to the fact that when $s=1$, both of the letters in an occurrence of $a(a+1)$ may be the first of their kind in the sequential representation of a partition.  To correct for this, formula \eqref{gfpartite2} would need to be changed to
$$\frac{qx^kE_k(x)}{1-x}\prod_{j=2}^{k-1}(q-1+E_j'(x)), \qquad k \geq 2,$$
with \eqref{gfpartite1} then adjusted appropriately, where $E_j(x)$ and $E_j'(x)$ are as in the proof of Theorem \ref{gfpartit}, but evaluated at $s=1$ (whence $j'=1$ and $j''=j-1$).  The preceding formula can be shown to equal the one found in \cite[Corollary~2.4]{MS1}, which is given by
$$\frac{x^k}{1-x\sum_{i=1}^k\frac{1-x^i(q-1)^i}{1-x(q-1)}}\prod_{j=1}^{k-1}\left(q-1+\frac{1-x^{j+1}(q-1)^{j+1}}{1-x(j+q)+x\left(\frac{1-x^{j+1}(q-1)^{j+1}}{1-x(q-1)}\right)}\right).$$
\medskip

We now seek an expression for the exponential generating function for the total number of occurrences of $a(a+s)$ on $\mathcal{P}_n$ and $\mathcal{P}_{n,k}$ for a fixed $s\geq 2$. To do so, let $P_k(x)=\frac{\partial}{\partial q}P_k(x,q)\mid_{q=1}$, and first note that by \eqref{totogf}, we have
$$(1-kx)P_k(x)-xP_{k-1}(x)=\frac{(1-sx)x^{k+1}}{\prod_{j=1}^k(1-jx)}, \qquad k \geq s+1,$$
with $P_0(x)=\cdots=P_s(x)=0$.

Now let $Q_k(x)=\sum_{n\geq0} [x^n](P_k(x))\cdot\frac{x^n}{n!}$ be the exponential generating function corresponding to $P_k(x)$. Then, by the facts $\sum_{n\geq k}S(n,k)x^n=\frac{x^k}{\prod_{j=1}^k(1-jx)}$ and $\sum_{n\geq k}S(n,k)\frac{x^n}{n!}=\frac{(e^x-1)^k}{k!}$, we have
$$\frac{d}{dx}Q_k(x)-kQ_k(x)-Q_{k-1}(x)=\frac{(e^x-1)^k}{k!}-s\int_0^x\frac{(e^t-1)^k}{k!}dt, \qquad k \geq s+1,$$
with $Q_0(x)=\cdots=Q_s(x)=0$.

Define $Q(x,y)=\sum_{k\geq s+1} Q_k(x)y^k$. Then
$$\frac{\partial}{\partial x}Q(x,y)-y\frac{\partial}{\partial y}Q(x,y)-yQ(x,y)=\sum_{k\geq s+1}\frac{(e^x-1)^ky^k}{k!}-s\int_0^x\sum_{k\geq s+1}\frac{(e^t-1)^ky^k}{k!}dt.$$
Define $Q'(x,y)=\frac{\partial}{\partial x}Q(x,y)$. Then
$$\frac{\partial}{\partial x}Q'(x,y)-y\frac{\partial}{\partial y}Q'(x,y)-yQ'(x,y)=\sum_{k\geq s+1}\frac{(e^x-1)^{k-1}e^xy^k}{(k-1)!}-s\sum_{k\geq s+1}\frac{(e^x-1)^ky^k}{k!}.$$
Define $H(v,u)=Q'(x,y)$ with $v=x$ and $u=e^vy$. Then
$$\frac{\partial}{\partial v}H(v,u)-ue^{-v}H(v,u)=\sum_{k\geq s+1}\frac{(e^v-1)^{k-1}e^{(1-k)v}u^k}{(k-1)!}-s\sum_{k\geq s+1}\frac{(e^v-1)^ke^{-kv}u^k}{k!},$$
which is equivalent to
$$\frac{\partial}{\partial v}\left(e^{ue^{-v}}H(v,u)\right)=e^{ue^{-v}}\sum_{k\geq s+1}\frac{(e^v-1)^{k-1}e^{(1-k)v}u^k}{(k-1)!}-se^{ue^{-v}}\sum_{k\geq s+1}\frac{(e^v-1)^ke^{-kv}u^k}{k!}.$$

Thus,
$$H(v,u)=e^{-ue^{-v}}\int_0^ve^{ue^{-t}}\left(\sum_{k\geq s+1}\frac{(e^t-1)^{k-1}e^{(1-k)t}u^k}{(k-1)!}-s\sum_{k\geq s+1}\frac{(e^t-1)^ke^{-kt}u^k}{k!}\right) dt,$$
which leads to
$$Q'(x,y)=e^{-y}\int_0^xe^{ye^{x-t}}\left(\sum_{k\geq s+1}\frac{(e^t-1)^{k-1}e^{t+k(x-t)}y^k}{(k-1)!}-s\sum_{k\geq s+1}\frac{(e^t-1)^ke^{k(x-t)}y^k}{k!}\right) dt.$$
Hence,
$$Q(x,y)=\int_0^x\int_0^re^{y(e^{r-t}-1)}\left(\sum_{k\geq s+1}\frac{(e^t-1)^{k-1}e^{t+k(r-t)}y^k}{(k-1)!}-s\sum_{k\geq s+1}\frac{(e^t-1)^ke^{k(r-t)}y^k}{k!}\right) dtdr.$$
By the fact $e^{y(e^x-1)}=\sum_{j\geq0}\frac{(e^x-1)^jy^j}{j!}$, we have
\begin{align*}
Q(x,y)&=\int_0^x\int_0^r\biggl(ye^r(e^{y(e^r-1)}-e^{y(e^{r-t}-1)}\sum_{k=0}^{s-1}\frac{(e^t-1)^ke^{k(r-t)}y^k}{k!})\\
&\qquad\qquad-s(e^{y(e^r-1)}-e^{y(e^{r-t}-1)}\sum_{k=0}^s\frac{(e^t-1)^ke^{k(r-t)}y^k}{k!})\biggr) dtdr\\
&=xe^{y(e^x-1)}-\int_0^x(1+sr)e^{y(e^r-1)}dr\\
&\qquad\qquad-\int_0^x\int_0^rye^{y(e^{r-t}-1)+r}\sum_{k=0}^{s-1}\frac{(e^t-1)^ke^{k(r-t)}y^k}{k!}dtdr\\
&\qquad\qquad+s\int_0^x\int_0^re^{y(e^{r-t}-1)}\sum_{k=0}^s\frac{(e^t-1)^ke^{k(r-t)}y^k}{k!}dtdr.
\end{align*}

We have the following result, where $exp_m(x)=\sum_{j=0}^m\frac{x^j}{j!}$.

\begin{theorem}\label{thqqs}
For all $s\geq2$,
\begin{align*}
Q(x,y)&=xe^{y(e^x-1)}-\int_0^x(1+sr)e^{y(e^r-1)}dr\\
&\quad-\int_0^x\int_0^re^{y(e^{r-t}-1)}\left(ye^rexp_{s-1}\left((e^t-1)e^{r-t}y\right)
-s\cdot exp_s\left((e^t-1)e^{r-t}y\right)\right)dtdr.
\end{align*}
\end{theorem}

By finding the coefficient of $y^k$ in $Q(x,y)$, we obtain
\begin{align*}
Q_k(x)&=\frac{x(e^x-1)^k}{k!}-\int_0^x\frac{(1+sr)(e^r-1)^k}{k!}dr\\
&\quad-\frac{1}{(k-1)!}\sum_{j=0}^{s-1}\int_0^x\int_0^r\binom{k-1}{j}e^r(e^{r-t}-1)^{k-1-j}(e^t-1)^je^{j(r-t)}dtdr\\
&\quad+\frac{s}{k!}\sum_{j=0}^s\int_0^x\int_0^r\binom{k}{j}(e^{r-t}-1)^{k-j}(e^t-1)^je^{j(r-t)}dtdr,
\end{align*}
for all $k\geq s+1$ and $s \geq 2$.

Using Theorem \ref{thqqs}, one can find explicit formulas for the total number of occurrences of $a(a+s)$ in all the members of $\mathcal{P}_n$ for a fixed $s$.  It is given by the coefficient of $x^n/n!$ in $Q(x,1)$, which will be denoted by $q_n$.

\begin{example}
Theorem \ref{thqqs} with $s=2$ gives
$$\frac{d}{dx}Q(x,1)=xe^{e^x+x-1}-(2x+3)e^{e^x-1}+(e^x+2)\int_0^xe^{e^t-1}dt+e^x+2.$$
By comparing the coefficients of $x^{n-1}/(n-1)!$ for $n\geq2$ on both sides of the last equation, we obtain
$$q_n=1+(n-4)B_{n-1}-(2n-4)B_{n-2}
+\sum_{j=1}^{n-1}\binom{n-1}{j}B_{j-1}.$$
By the fact $1+\sum_{j=1}^{n-1}\binom{n-1}{j}B_{j-1}=\sum_{j=0}^{n-1}B_j$, this gives
\begin{equation}\label{s=2Qxy}
q_n=(n-3)B_{n-1}-(2n-5)B_{n-2}+\sum_{j=0}^{n-3}B_{j}, \qquad n \geq 2.
\end{equation}
\end{example}

\begin{example}
Theorem \ref{thqqs} with $s=3$ gives
$$\frac{d}{dx}Q(x,1)=\frac{2x-1}{2}e^{e^x+x-1}-\frac{6x+11}{2}e^{e^x-1}
+\frac{e^{2x}+4e^x+6}{2}\int_0^xe^{e^t-1}dt
+\frac{1}{2}e^{2x}+2e^x+\frac{7}{2},$$
which implies that $q_n$ for $n \geq 2$ is given by
$$-\frac{1}{2}B_n+\frac{2n-13}{2}B_{n-1}-3(n-2)B_{n-2}+2^{n-2}+2
+\sum_{j=1}^{n-1}\binom{n-1}{j}(2^{n-2-j}+2)B_{j-1}.$$
\end{example}

\begin{example}
Theorem \ref{thqqs} with $s=4$ gives
\begin{align*}
\frac{d}{dx}Q(x,1)&=-\frac{1}{6}e^{e^x+2x-1}+\frac{6x-7}{6}e^{e^x+x-1}-\frac{12x+25}{3}e^{e^x-1}
+\frac{e^{3x}+6e^{2x}+21e^x+30}{6}\\
&\quad+\frac{e^{3x}+6e^{2x}+18e^x+24}{6}\int_0^xe^{e^t-1}dt,
\end{align*}
which implies that $q_{n+1}$ for $n \geq 1$ is given by
$$-\frac{1}{6}B_{n+2}-B_{n+1}+\frac{3n-25}{3}B_n
-4(n-1)B_{n-1}+\frac{3^n+6\cdot2^n+21}{6}+\sum_{j=1}^n\binom{n}{j}\frac{3^{n-j}+6\cdot2^{n-j}+18}{6}B_{j-1}.$$
\end{example}

In Section \ref{combproofs}, a combinatorial proof is provided for \eqref{s=2Qxy} which makes use of the fact that $q_n$ when $s=2$ equals the total number of occurrences $a(a+2)$ in all the members of $\mathcal{P}_n$.

\section{Distribution of $a(a\pm s)$ on $k$-ary words}\label{apm s}

Given positive integers $k$ and $s$ with $k \geq s+1$, let $b_n=b_n(q)=b_n^{(k,s)}(q)$ denote the distribution for the number of occurrences of the string $a(a\pm s)$ on the set of $k$-ary words of length $n$ for $n \geq 1$, with $b_0=1$. Further, given $1 \leq i \leq k$, let $b_{n,i}=b_{n,i}(q)$ denote the restriction of $b_n$ to words ending in $i$.

To write recurrences for $b_{n,i}$, we must differentiate the following cases on $k$ and $s$.

\subsection{The case $s+1 \leq k \leq 2s$}

Upon reasoning as in Lemma \ref{lem1} above, we have in this case the following recurrences for $n \geq 2$:
\begin{align}
b_{n,i}&=b_{n-1}+(q-1)b_{n-1,i+s}, \qquad 1 \leq i \leq k-s, \label{bneq1}\\
b_{n,i}&=b_{n-1}, \qquad k-s+1 \leq i \leq s,\label{bneq2}\\
b_{n,i}&=b_{n-1}+(q-1)b_{n-1,i-s}, \qquad s+1 \leq i \leq k, \label{bneq3}
\end{align}
with $b_{1,i}=1$ for all $i$.

To aid in solving \eqref{bneq1}--\eqref{bneq3}, we define the generating functions $B_i(x)=\sum_{n\geq1}b_{n,i}x^n$ for $1 \leq i \leq k$ and $B(x)=\sum_{n\geq0}b_nx^n=1+\sum_{i=1}^kB_i(x)$. We have the following explicit formula for $b_n$ in this case.

\begin{theorem}\label{th2}
Let  $s+1\leq k\leq 2s$.  If $n\geq 1$, then
\begin{align*}
b_n&=k((2s-k)(q-1))^{(n-1)/2}U_{n-1}\left(\frac{k+q-1}{2\sqrt{(2s-k)(q-1)}}\right)\\
&\qquad\qquad-((2s-k)(q-1))^{n/2}U_{n-2}\left(\frac{k+q-1}{2\sqrt{(2s-k)(q-1)}}\right),
\end{align*}
where $U_n(t)$ is the $n$-th Chebyshev polynomial of the second kind with $U_{-1}(t)=0$.
\end{theorem}
\begin{proof}
Multiplying both sides of \eqref{bneq1}--\eqref{bneq3} by $x^n$, and summing over $n\geq2$, we obtain
\begin{align*}
B_i(x)&=xB(x)+(q-1)xB_{i+s}(x), \qquad 1 \leq i \leq k-s, \\ B_i(x)&=xB(x), \qquad k-s+1 \leq i \leq s,\\
B_i(x)&=xB(x)+(q-1)xB_{i-s}(x), \qquad s+1 \leq i \leq k. \end{align*}
Note further that if $1\leq i\leq k-s$ or $s+1\leq i\leq k$, then
$$B_i(x)=xB(x)+(q-1)x(xB(x)+(q-1)xB_i(x)),$$
which implies
$$B_1(x)=\cdots=B_{k-s}(x)=B_{s+1}(x)=\cdots=B_k(x)=\frac{xB(x)}{1-(q-1)x}.$$
Hence,
$$B(x)=1+(2s-k)xB(x)+2(k-s)\frac{xB(x)}{1-(q-1)x},$$
which leads to
\begin{equation}\label{Bxeqn}
B(x)=\frac{1+(1-q)x}{1+(1-q-k)x-(2s-k)(1-q)x^2}.
\end{equation}
Upon recalling the well-known generating function formula $\sum_{n\geq0}U_n(t)x^n=\frac{1}{1-2tx+x^2}$, the proof may be completed
\end{proof}

From \eqref{Bxeqn}, we have that $b_n$ in the case $s+1 \leq k \leq 2s$ satisfies the two-term recursion
\begin{equation}\label{bnrec}
b_n=(k-1+q)b_{n-1}+(1-q)(2s-k)b_{n-2}, \qquad n \geq 2,
\end{equation}
with $b_0=1$ and $b_1=k$.

It is instructive to provide a combinatorial explanation of \eqref{bnrec}, which is done in Section \ref{combproofs} below.

\subsection{The case $k \geq 2s+1$}

Here, we have the recurrences for $n \geq 2$,
\begin{align}
b_{n,i}&=b_{n-1}+(q-1)b_{n-1,i+s}, \qquad 1 \leq i \leq s, \label{bn'eq1}\\
b_{n,i}&=b_{n-1}+(q-1)(b_{n-1,i-s}+b_{n-1,i+s}), \qquad s+1 \leq i \leq k-s,\label{bn'eq2}\\
b_{n,i}&=b_{n-1}+(q-1)b_{n-1,i-s}, \qquad k-s+1 \leq i \leq k, \label{bn'eq3}
\end{align}
with $b_{1,i}=1$ for all $i$.

Let $B_i(x)$ for $1 \leq i \leq k$ and $B(x)$ be as before.

\begin{theorem}\label{th2}
Let $k=ds+r$, where $d\geq2$ and $1\leq r\leq s$. Then
\begin{align*}
B(x)&=\frac{1}{1-rH_d(x)-(s-r)H_{d-1}(x)},
\end{align*}
where
$$H_d(x)=\sum_{\ell=0}^d\frac{x^2(1-q)\left(U_{\ell+1}(\frac{1}{2x(1-q)})+ U_\ell(\frac{1}{2x(1-q)})
+(-1)^\ell\right)^2}{(1+2x(1-q))^2U_{\ell+1}(\frac{1}{2x(1-q)}) U_{\ell}(\frac{1}{2x(1-q)})}$$
and $U_n(t)$ denotes the $n$-th Chebyshev polynomial of the second kind.
\end{theorem}
\begin{proof}
Multiplying both sides of \eqref{bn'eq1}--\eqref{bn'eq3} by $x^n$, and summing over $n\geq2$, we obtain
\begin{align*}
B_i(x)&=xB(x)+(q-1)xB_{i+s}(x), \qquad 1 \leq i \leq s,\\
B_i(x)&=xB(x)+(q-1)x(B_{i-s}(x)+B_{i+s}(x)), \qquad s+1 \leq i \leq k-s,\\
B_i(x)&=xB(x)+(q-1)xB_{i-s}(x), \qquad k-s+1 \leq i \leq k.
\end{align*}
Let $k=ds+r$ be as above. Then the preceding recurrences may be written as
\begin{align*}
B_i(x)&=xB(x)+(q-1)xB_{i+s}(x),\\
B_{s+i}(x)&=xB(x)+(q-1)x(B_i(x)+B_{2s+i}(x)),\\
B_{2s+i}(x)&=xB(x)+(q-1)x(B_{s+i}(x)+B_{3s+i}(x)),\\
\vdots&\\
B_{(d-1)s+i}(x)&=xB(x)+(q-1)x(B_{(d-2)s+i}(x)+B_{ds+i}(x)),\\
B_{ds+i}(x)&=xB(x)+(q-1)xB_{(d-1)s+i}(x),
\end{align*}
for $1\leq i\leq r$, with the same system holding for $r+1\leq i\leq s$, but with $d$ replaced by $d-1$.  Let us first focus on the case $1 \leq i \leq r$. Define the column vectors
$$v=(B_i(x),B_{s+i}(x),\ldots,B_{ds+i}(x))^t \text{ and } b=xB(x)(1,1,\ldots,1)^t.$$ Define the $(d+1)\times(d+1)$ matrix ${\bf B}=(a_{i,j})_{0\leq i,j\leq d}$, where $a_{i,i}=1$, $a_{i,j}=x(1-q)$ if $|i-j|=1$ and $a_{i,j}=0$ if $|i-j|>1$.

Then the above system can be written as ${\bf B}v=b$. Note the LU-decomposition ${\bf B}={\bf LU}$, where ${\bf L}=(\ell_{i,j})$ with $$\ell_{i,i}=1,\quad \ell_{i,i-1}=\frac{U_{i-1}(\frac{1}{2x(1-q)})}{U_i(\frac{1}{2x(1-q)})}\quad \mbox{and $\ell_{i,j}=0$ whenever $j-i\geq 1$ or $i-j\geq2$},$$
and ${\bf U}=(u_{i,j})$ with
$$u_{i,i+1}=x(1-q),\quad u_{i,i}=\frac{x(1-q)U_{i+1}(\frac{1}{2x(1-q)})}{U_i(\frac{1}{2x(1-q)})}\quad \mbox{and $u_{i,j}=0$ whenever $j-i\geq 2$ or $i-j\geq1$}.$$
The solution of the system ${\bf L}z=b$, where $z=(z_0,\ldots,z_d)^t$, is given by
$$z_j=\frac{xB(x)}{U_j(\frac{1}{2x(1-q)})}\sum_{k=0}^j(-1)^kU_{j-k}(\frac{1}{2x(1-q)}),$$
for all $j=0,1,\ldots,d$, which can be shown by induction on $j$, starting with $j=0$. The solution of the system ${\bf U}v=z$ is given by
$$B_{js+i}(x)=B(x)U_j(\frac{1}{2x(1-q)})
\sum_{\ell=j}^d\sum_{k=0}^{\ell}\frac{(-1)^{\ell+k-j}
U_{\ell-k}(\frac{1}{2x(1-q)})}{(1-q)U_{\ell+1}(\frac{1}{2x(1-q)}) U_{\ell}(\frac{1}{2x(1-q)})},
$$
for all $j=0,1,\ldots,d$, which can be shown by induction on $j$, starting with $j=d$.  It is seen that the same formula also holds for $B_{js+i}(x)$ where $r+1 \leq i \leq s$ and $0 \leq j \leq d-1$, but with $d$ replaced by $d-1$. Hence, by the fact $$B(x)=1+\sum_{j=0}^d\sum_{i=1}^rB_{js+i}(x)+\sum_{j=0}^{d-1}\sum_{i=r+1}^sB_{js+i}(x),$$ we obtain
$$B(x)=1+rB(x)H_d(x)+(s-r)B(x)H_{d-1}(x),$$
where
$$H_d(x)=\sum_{j=0}^d\sum_{\ell=j}^d\sum_{k=0}^{\ell}\frac{(-1)^{\ell+k-j}
U_j(\frac{1}{2x(1-q)})U_{\ell-k}(\frac{1}{2x(1-q)})}{(1-q)U_{\ell+1}(\frac{1}{2x(1-q)}) U_{\ell}(\frac{1}{2x(1-q)})}.$$

We now show that the preceding formula for $H_d(x)$ matches the one given in the theorem statement above, from which the stated result will follow.  Interchanging summation, we have
\begin{align*}
H_d(x)&=\sum_{\ell=0}^d\sum_{j=0}^{\ell}\sum_{k=0}^{\ell}\frac{(-1)^{\ell+k-j}
U_j(\frac{1}{2x(1-q)})U_{\ell-k}(\frac{1}{2x(1-q)})}{(1-q)U_{\ell+1}(\frac{1}{2x(1-q)}) U_{\ell}(\frac{1}{2x(1-q)})}\\
&=\sum_{\ell=0}^d\frac{\sum_{j=0}^{\ell}(-1)^jU_j(\frac{1}{2x(1-q)})
\sum_{k=0}^{\ell}(-1)^{\ell-k}U_{\ell-k}(\frac{1}{2x(1-q)}) }{(1-q)U_{\ell+1}(\frac{1}{2x(1-q)}) U_{\ell}(\frac{1}{2x(1-q)})}\\
&=\sum_{\ell=0}^d\frac{\left(\sum_{j=0}^{\ell}(-1)^jU_j(\frac{1}{2x(1-q)})
\right)^2}{(1-q)U_{\ell+1}(\frac{1}{2x(1-q)}) U_{\ell}(\frac{1}{2x(1-q)})}.
\end{align*}
Applying now the identity
$$\sum_{j=0}^n(-1)^jU_j(t)=\frac{(-1)^n(U_{n+1}(t)+U_n(t))+1}{2(1+t)}, \qquad n \geq 0,$$
yields the desired formula for $H_d(x)$.
\end{proof}

\section{Combinatorial proofs}\label{combproofs}

In this section, we provide combinatorial explanations of several of the prior results shown above by algebraic methods. We first provide such proofs of the recurrences \eqref{analtrec} and \eqref{analtrecq=0}.

{\bf Proofs of recurrences \eqref{analtrec} and \eqref{analtrecq=0}:}\medskip

Let $\mathcal{A}_n=\mathcal{A}_{n,k}$ denote the set of $k$-ary words of length $n$.  To show \eqref{analtrec}, we first interpret the right-hand side as a certain weighted sum of members of $\mathcal{A}_n$, where $q$ is an indeterminate.  Let $\mathcal{A}_n^{(j)}$ denote the subset of $\mathcal{A}_n$ whose members end in an arithmetic progression of length at least $j+1$ having common difference $s$.  Given $\pi \in \mathcal{A}_n$, let $\text{wght}(\pi)$ be given by $q^{\mu(\pi)}$, where $\mu(\pi)$ denotes the number of strings of the form $a(a+s)$, and hence $a_n=\sum_{\pi \in \mathcal{A}_n}\text{wght}(\pi)$, by definition.  Define $\text{wght}_j(\pi)$ by
\begin{equation}\label{wghtj}
\text{wght}_j(\pi)=\begin{cases}
      \text{wght}(\pi)/q^j, & \text{if $\pi \in \mathcal{A}_n^{(j)}-\mathcal{A}_n^{(j+1)}$;}\\
         \text{wght}(\pi)/q^{j+1}, & \text{if  $\pi \in \mathcal{A}_n^{(j+1)}$;}\\
         \text{wght}(\pi), & \text{otherwise.}
    \end{cases}
\end{equation}
Note that
\begin{equation}\label{wghtje2}
\sum_{\pi \in \mathcal{A}_n^{(j)}}\text{wght}_j(\pi)=(k-js)a_{n-j-1}, \qquad 0 \leq j \leq k''.
\end{equation}
To realize \eqref{wghtje2}, consider appending $u(u+s)\cdots(u+js)$ for some $u \in [k-js]$ to  $\rho' \in \mathcal{A}_{n-j-1}$ to obtain $\rho \in \mathcal{A}_n^{(j)}$.  Then one can show $\rho$ is weighted by $\text{wght}_j(\rho)$ in $(k-js)a_{n-j-1}$, by considering cases on whether or not the last letter of $\rho'$ equals $u-s$ and using definition \eqref{wghtj}.

Now consider the sum of weights $\sum_{j=0}^{k''}(q-1)^j(k-js)a_{n-j-1}~(*)$ and suppose $\pi \in \mathcal{A}_n^{(\ell)}-\mathcal{A}_n^{(\ell+1)}$ for some $0 \leq \ell \leq k''$.  We wish to show that the contribution of $\pi$ towards the sum of weights (*) is given by $\text{wght}(\pi)$.  To do so, we need only consider terms for $0 \leq j \leq \ell$ in (*), since $\pi \in \mathcal{A}_n^{(\ell)}-\mathcal{A}_n^{(\ell+1)}$ implies it does not contribute to any of the terms for $j>\ell$, as they involve sums of weights of members of $\mathcal{A}_n^{(j)}$ for $j>\ell$.  By \eqref{wghtje2}, it then suffices to show
\begin{equation}\label{wghtje3}
\sum_{j=0}^\ell(q-1)^j\text{wght}_j(\pi)=\text{wght}(\pi), \qquad \pi \in \mathcal{A}_n^{(\ell)}-\mathcal{A}_n^{(\ell+1)}.
\end{equation}

Note that \eqref{wghtje3} may be shown using the definition of $\text{wght}_j(\pi)$ in \eqref{wghtj} and making use of the formula for a finite geometric series and some algebra.  To give a more combinatorial proof, first observe that
$$\sum_{j=0}^\ell(q-1)^j\text{wght}_j(\pi)=\sum_{j=0}^{\ell-1}\frac{(q-1)^j}{q^{j+1}}\text{wght}(\pi)+\frac{(q-1)^\ell}{q^{\ell}}\text{wght}(\pi)$$
is a sum of various weights for $\pi$ wherein the first $n-\ell-1$ adjacencies are assigned weights of $q$ or $1$ according to $\text{wght}(\pi)$ as usual, whereas the final $\ell$ adjacencies of $\pi$ are assigned weights of $1$, $-1$ or $q$ according to the rules: (i) at most a single weight of $1$ is used and (ii) if a weight of $1$ is used, any adjacencies prior to it must be assigned a $q$.  To see this, divide $q^{j+1}$ if $0 \leq j < \ell$ into $\text{wght}(\pi)$ to cancel out the $q$ weights corresponding to the final $j+1$ adjacencies and then reassign the final $j$ adjacencies weights of $-1$ and $q$ in an arbitrary fashion. On the other hand, if $j=\ell$, then we cancel only the $q$ weights for the final $\ell$ adjacencies before reassigning weights.  Note that when $0 \leq j<\ell$, there is always a single adjacency amongst the final $\ell$ that receives a weight of $1$.

We may pair these various weights for $\pi$ by considering the leftmost adjacency among the final $\ell$ that is assigned a weight of $1$ or $-1$ and switching to the other option.  This leaves only the weight in which the final $\ell$ adjacencies are each assigned $q$, which is in accordance with what it would be when computing $\text{wght}(\pi)$.  Since the first $n-\ell-1$ adjacencies of $\pi$ were assigned weights in the usual way using $\text{wght}(\pi)$, the left-hand side of \eqref{wghtje3} works out to $\text{wght}(\pi)$, as desired.  Since \eqref{wghtje3} holds for all $0 \leq \ell \leq k''$, we see that $\pi$ contributes $\text{wght}(\pi)$ towards the sum (*) for all $\pi \in \mathcal{A}_n$, and hence the sum is given by $a_n$, which implies \eqref{analtrec}.

We now show \eqref{analtrecq=0}.  Let $\mathcal{D}_n$ denote the subset of $\mathcal{A}_n$ whose members contain no adjacencies of the form $a(a+s)$ for some $a$.  Given $1 \leq j \leq k''$, let $\mathcal{D}_n^{(j)}$ denote the subset of $\mathcal{A}_n$ whose members end in an arithmetic progression of length $j+1$ with common difference $s$, but otherwise contain no other occurrences of $a(a+s)$, with $\mathcal{D}_n^{(0)}=\mathcal{D}_n$ and $\mathcal{D}_n^{(k''+1)}=\varnothing$.  Observe that $\pi \in \mathcal{D}_n^{(j)}\cup\mathcal{D}_n^{(j+1)}$ for some $0 \leq j \leq k''$ if and only if $\pi$ can be expressed as
\begin{equation}\label{formofpi}
\pi=\pi'u(u+s)\cdots (u+js),
\end{equation}
where $\pi' \in \mathcal{D}_{n-j-1}$ and $u \in [k-js]$, upon considering whether or not the final letter of $\pi'$ equals $u-s$. Further, since the final letter of $\pi'$ and $u$ may be chosen independently of one another, there are $(k-js)a_{n-j-1}(0)$ possible $\pi$ of the form in \eqref{formofpi}.

We thus have
\begin{align*}
&\sum_{j=0}^{k''}(-1)^j(k-js)a_{n-j-1}(0)=\sum_{j=0}^{k''}(-1)^j(|\mathcal{D}_n^{(j)}|+|\mathcal{D}_{n}^{(j+1)}|)\\
&=|\mathcal{D}_n^{(0)}|+\sum_{j=1}^{k''}(-1)^j|\mathcal{D}_n^{(j)}|+\sum_{j=1}^{k''+1}(-1)^{j-1}|\mathcal{D}_n^{(j)}|\\
&=|\mathcal{D}_n^{(0)}|+\sum_{j=1}^{k''}(-1)^j|\mathcal{D}_n^{(j)}|-\sum_{j=1}^{k''}(-1)^j|\mathcal{D}_n^{(j)}|=|\mathcal{D}_n|=a_n(0),
\end{align*}
which gives \eqref{analtrecq=0}.  Alternatively, it is possible to make the proof completely combinatorial by considering a further structure as follows.  If $1 \leq j \leq k''$, then let $\mathcal{E}_n^{(j)}$ denote the set of ordered pairs $\rho=(\lambda,i)$, where $\lambda \in \mathcal{D}_n^{(j)}$ and $i$ equals $j-1$ or $j$, with $\mathcal{E}_n^{(0)}$ the set of pairs $\rho=(\lambda,0)$, where $\lambda \in \mathcal{D}_n^{(0)}$.  Define the sign of $\rho=(\lambda,i) \in \mathcal{E}_n^{(j)}$ as $(-1)^i$.  Then it is seen that the right side of \eqref{analtrecq=0} gives the sum of the signs of all members of $\mathcal{E}_n=\cup_{j=0}^{k''}\mathcal{E}_n^{(j)}$.  Define a sign-reversing involution on $\mathcal{E}_n$ by changing $i$ in $\rho=(\lambda,i) \in \mathcal{E}_n^{(j)}$ for some $j>0$ to the other option.  The set of survivors of this involution consists of the members of $\mathcal{E}_n^{(0)}$, of which there are $a_n(0)$ with each having positive sign, which completes the proof.\hfill \qed \medskip

For the particular cases of $a_n(0)$ from Corollary \ref{cor3}, it is possible also to provide a more direct combinatorial explanation, which makes use of a subtraction argument.

{\bf Proof of Corollary \ref{cor3}:} \medskip

For the first equality, one may assume $n \geq 2$, since it clearly holds for $n=0,1$. Then $a_n^{(3,2)}(0)$ is seen to satisfy the recurrence $u_n=3u_{n-1}-u_{n-2}$ as follows.  Let $\mathcal{J}_n$ denote the subset of $\mathcal{A}_{n,3}$ whose members contain no occurrences of the string 1-3 and let $u_n=|\mathcal{J}_n|$.  Then there are clearly $2u_{n-1}$ members of $\mathcal{J}_n$ that end in $1$ or $2$.  Words in $\mathcal{J}_n$ ending in $3$ may be obtained by appending $3$ to words in $\mathcal{J}_{n-1}$ not ending in $1$, of which there are $u_{n-1}-u_{n-2}$, by subtraction.  Combining the prior cases then gives the desired recurrence for $u_n$.  A similar argument shows $u_n=a_n^{(4,2)}(0)$ satisfies $u_n=4u_{n-1}-2u_{n-2}$ for $n \geq 2$.

For the third formula in Corollary \ref{cor3}, we may assume $n \geq 3$, upon verifying the equality directly for $0 \leq n \leq 2$, and then show that $a_n^{(5,2)}(0)$ satisfies the recurrence $u_n=5u_{n-1}-3u_{n-2}+u_{n-3}$ for $n \geq 3$.  Let $u_n=|\mathcal{L}_n|$, where $\mathcal{L}_n$ denotes the subset of $\mathcal{A}_{n,5}$ whose members do not contain the strings 1-3, 2-4 and 3-5. Note that there are clearly $2u_{n-1}$ members of $\mathcal{L}_n$ that end in $1$ or $2$.  By subtraction, there are $2(u_{n-1}-u_{n-2})$ words in $\mathcal{L}_n$ that end in $3$ or $4$, upon appending $3$ or $4$ respectively to words in $\mathcal{L}_{n-1}$ not ending in $1$ or $2$.  To complete the proof, we must enumerate the subset of $\mathcal{L}_n$ whose members end in $5$, or equivalently the subset of $\mathcal{L}_{n-1}$ whose members do not end in $3$.  By subtraction, there are $u_{n-2}-u_{n-3}$ members of $\mathcal{L}_{n-1}$ ending in $3$, upon appending $3$ to members of $\mathcal{L}_{n-2}$ not ending in $1$.  Hence, again by subtraction, there are $u_{n-1}-(u_{n-2}-u_{n-3})$ members of $\mathcal{L}_n$ not ending in $5$.  Combining this case with the previous yields the desired recurrence for $u_n$ and completes the proof. \hfill \qed \medskip

{\bf Proofs of Corollaries \ref{Fnpolycor1} and \ref{Fnpolycor2}:} \medskip

We may assume $n \geq 2$ throughout.  Note first that $j_n(0,1)=|\mathcal{J}_n'|$, where $\mathcal{J}_n'$ denotes the subset of $\mathcal{J}_n$ whose members have no levels.  Let $\mathcal{J}_n''$ denote the subset of $\mathcal{J}_n'$ whose members start with $2$ and $t_n=|\mathcal{J}_n''|$. Then there are $t_{n-1}$ members of $\mathcal{J}_n'$ that start with $1$, and $t_n$ that start with $3$. To realize the latter, consider replacing all $2$'s by $3$'s and vice versa prior to the first $1$ (if it exists), which defines a bijection between the subsets of $\mathcal{J}_n'$ starting with $2$ and $3$.  Thus, in order to establish $|\mathcal{J}_n'|=F_{n+3}$, it is enough to show $t_n=F_{n+1}$.  To do so, we make use of the combinatorial interpretation of $F_{n+1}$ as the enumerator of the set $\mathcal{F}_{n}$ of square-and-domino tilings of length $n$ (see, e.g., \cite[Chapter~1]{BQ}) and seek a bijection between $\mathcal{J}_n''$ and $\mathcal{F}_n$.

Let $\pi=\pi_1\cdots \pi_n \in \mathcal{J}_n''$.  We first overline both letters in the rightmost occurrence in $\pi$ of 2-1 or 3-2, if it exists.  Once this is done, we then overline both letters in the rightmost occurrence of 2-1 or 3-2 occurring to the left of the first overlined string (where two overlined strings cannot have a letter in common).  We then continue from right to left overlining disjoint strings of 2-1 or 3-2 until one reaches the beginning of $\pi$.  Beneath each of the remaining letters that is not part of an overlined string, we place a dot.  Finally, we put a square for each dot and a domino for each overlined string going from left to right to obtain a member $\pi' \in \mathcal{F}_n$.  One may verify that the mapping $\pi \mapsto \pi'$ yields the desired bijection between $\mathcal{J}_n''$ and $\mathcal{F}_n$, which establishes the first statement in Corollary \ref{Fnpolycor1}.  For the second, observe that $j_n(1,0)=\binom{n+2}{2}$, the enumerated members of $\mathcal{J}_n$ in this case consisting of the weakly decreasing $3$-ary sequences of length $n$.

For the first formula in Corollary \ref{Fnpolycor2}, note that a level may be obtained by inserting directly following any letter in $\pi=\pi_1\cdots \pi_{n-1} \in \mathcal{J}_{n-1}$ an extra copy of it and marking the resulting level.  From this, it is seen that there are $(n-1)F_{2n}$ marked members of $\mathcal{J}_n$ wherein a level is marked, which implies the first formula.  For the second, note first that $\lambda \in \mathcal{J}_n$ in which an ascent is marked may be decomposed as $\lambda=\lambda'\underline{x}\lambda''$, where the marked ascent $x$ is underlined and may be either 1-2 or 2-3.  Further, $\lambda' \in \mathcal{J}_i$ and $\lambda'' \in \mathcal{J}_{n-i-2}$, as the letters in $x$ do not restrict the choice of either $\lambda'$ or $\lambda''$. Considering all possible $i$, and recalling $|\mathcal{J}_i|=F_{2i+2}$, implies that there are
$$2\sum_{i=0}^{n-2}F_{2i+2}F_{2n-2i-2}=2\sum_{i=1}^{n-1}F_{2i}F_{2n-2i}$$
members of $\mathcal{J}_n$ in which some ascent is marked.  The formula for the total number of ascents in $\mathcal{J}_n$ now follows from the identity
$$5\sum_{i=1}^{n-1}F_{2i}F_{2n-2i}=(n-1)L_{2n}-2F_{2n-2},  \qquad n \geq 1.$$
This formula can be shown most easily by making use of the Binet formulas for $F_n$ and $L_n$.  Alternatively, it is possible to provide a combinatorial proof utilizing the tiling interpretations for $F_n$ and $L_n$ that extends the argument given for the identity
$$5\sum_{i=1}^{n-1}F_iF_{n-i}=nL_n-F_n, \qquad n \geq 1,$$
in \cite[Identity~58]{BQ}. \hfill \qed \medskip

{\bf Proof of Corollary \ref{totPnk}:}\medskip

The total in question equals the cardinality of the set $\mathcal{P}_{n,k}^*$ whose members are derived from those in $\mathcal{P}_{n,k}$ by marking a single occurrence of $a(a+s)$ for some $a$.  Let $\pi=\pi_1\cdots\pi_n \in \mathcal{P}_{n,k}$.  By an \emph{initial} letter within $\pi$, we mean one corresponding to the first of its kind within a left-to-right scan of $\pi$.  Then there are $(k-s)S(n-1,k)$ members of $\mathcal{P}_{n,k}^*$ in which the second letter in the marked string is initial.  To see this, select some $x \in [s+1,k]$, insert an $x-s$ directly prior to the first $x$  within a member of $\mathcal{P}_{n-1,k}$ and then mark the resulting occurrence of $a(a+s)$.  Note that $s \geq 2$ implies there are no occurrences of $a(a+s)$ in which the $a$ is initial.

Now suppose that neither letter of the marked string within a member of $\mathcal{P}_{n,k}^*$ is initial.  Let $\pi=\pi^{(1)}\cdots \pi^{(k)}$, where each $\pi^{(i)}$ is $i$-ary and starts with $i$.  Suppose that the marked string is contained within $\pi^{(j)}$ for some $j \in [s+1,k]$.  Then we may write $\pi^{(j)}=j\alpha\underline{a(a+s)}\beta$, where the marked string is underlined and the subsequences $\alpha$ and $\beta$ are possibly empty $j$-ary words.  Suppose $|\beta|=t$ for some $t \geq 0 $.  Note that there are $j-s$ choices for $a$ and $j^t$ possibilities for $\beta$.  Deleting the section $\underline{a(a+s)}\beta$ from $\pi^{(j)}$ then results in an arbitrary member of $\mathcal{P}_{n-t-2,k}$.  Thus, considering all possible $j$ and $t$ gives $\sum_{j=s+1}^k\sum_{t=0}^{n-k-2}S(n-t-2,k)j^t(j-s)$ members of $\mathcal{P}_{n,k}^*$ in which the marked string does not contain an initial letter.  Combining this with the first case above then yields $|\mathcal{P}_{n,k}^*|$, and hence the desired formula for the total number of occurrences of $a(a+s)$. \hfill \qed \medskip

{\bf Proof of formula \eqref{s=2Qxy}:} \medskip

Formula \eqref{s=2Qxy} clearly holds when $n=2$ or 3, so we may assume $n \geq 4$. Let $\pi=\pi_1\cdots \pi_n \in \mathcal{P}_{n,k}$.  By an \emph{initial} letter within $\pi$, we mean one corresponding to the leftmost occurrence of its kind, with all other letters being \emph{non-initial}.  Likewise, we will describe an occurrence of $a(a+2)$ for some $a$ in which the $a+2$ is the first letter of its kind within $\pi$ as \emph{initial}, with all other occurrences being \emph{non-initial}. Note that in an adjacency of the form $a(a+2)$, the $a$ cannot be initial due to the restricted growth condition on set partitions.  We first seek to count all non-initial occurrences of $a(a+2)$ within $\mathcal{P}_{n,k}$, where $k \geq 2$ and $n \geq k+2$. One may form such an occurrence as follows.  Let $\pi=\pi^{(1)}\cdots \pi^{(k)} \in \mathcal{P}_{n-1,k}$, where each $\pi^{(i)}$ is $i$-ary and starts with $i$, and choose some $j \in [2,k]$ such that $\pi^{(j)}=j\tau$ where $\tau \neq \varnothing$. Then inserting $\ell+2$ directly following any letter $\ell$ of $\tau$ such that $\ell\neq j-1,j$ produces a non-initial $a(a+2)$, with all such occurrences of $a(a+2)$ arising (uniquely) in this manner as $\pi$ and $j$ vary.

We now enumerate all letters $t$ in $\mathcal{P}_{n-1,k}$ such that $t \in \pi^{(j)}$ for some $2 \leq j \leq k$ within some member of $\mathcal{P}_{n-1,k}$, where either $t=j-1$ or $t=j$ and is non-initial.  Note that such a letter $t$ arises by inserting either $j-1$ or $j$  directly after some letter $x$ within a member $\lambda=\lambda^{(1)}\cdots\lambda^{(k)} \in \mathcal{P}_{n-2,k}$, where $x \in \lambda^{(j)}$ and $j \geq 2$.  Therefore, the number of such letters $t$ is twice the number of letters $x$, and we seek to enumerate the latter.  To do so, note that there are clearly $(n-2)S(n-2,k)$ letters altogether in $\mathcal{P}_{n-2,k}$, and from these, we must subtract those occurring as part of $\lambda^{(1)}$ within some $\lambda \in \mathcal{P}_{n-2,k}$.  There are $S(n-i-2,k)$ such letters occurring in the $(i+1)$-st position for some $0 \leq i \leq n-k-2$, and thus $\sum_{i=0}^{n-k-2}S(n-i-2,k)$ altogether within the members of $\mathcal{P}_{n-2,k}$.  Hence, there are $(n-2)S(n-2,k)-\sum_{i=0}^{n-k-2}S(n-i-2,k)$ letters $x$ as described.

We then subtract the number of letters $t$ from the total number of non-initial letters $>1$ in all the members of $\mathcal{P}_{n-1,k}$ to obtain the number of non-initial occurrences of $a(a+2)$.  Let $\rho=\rho^{(1)}\cdots \rho^{(k)}$, where each $\rho^{(i)}$ is $i$-ary and starts with $i$,  denote an arbitrary member of $\mathcal{P}_{n-1,k}$.  By subtraction of the number of $1$'s, other than the first, which occur as part of $\rho^{(1)}$ within some $\rho$, we have that there are $(n-k-1)S(n-1,k)-\sum_{i=0}^{n-k-2}S(n-i-2,k)$ non-initial letters $>1$ in $\mathcal{P}_{n-1,k}$.  Thus, there are
\begin{align*}
&(n-k-1)S(n-1,k)-\sum_{i=0}^{n-k-2}S(n-i-2,k)-2((n-2)S(n-2,k)-\sum_{i=0}^{n-k-2}S(n-i-2,k))\\
&=(n-k-1)S(n-1,k)-2(n-2)S(n-2,k)+\sum_{i=0}^{n-k-2}S(n-i-2,k) \quad (*)
\end{align*}
non-initial occurrences of $a(a+2)$ in members of $\mathcal{P}_{n,k}$ for $2 \leq k \leq n-2$.  Note that (*) also gives the correct value of zero when $k=1$ and $k=n-1$. Summing (*) over all $1 \leq k \leq n-1$, and making use of the identity $\sum_{k=1}^{n-1}kS(n-1,k)=B_n-B_{n-1}$, which may be explained combinatorially, then yields
\begin{align*}
&(n-1)\sum_{k=1}^{n-1}S(n-1,k)-\sum_{k=1}^{n-1}kS(n-1,k)-2(n-2)\sum_{k=1}^{n-2}S(n-2,k)+\sum_{k=1}^{n-2}\sum_{i=0}^{n-k-2}S(n-i-2,k)\\
&=(n-1)B_{n-1}-(B_n-B_{n-1})-2(n-2)B_{n-2}+\sum_{i=0}^{n-3}\sum_{k=1}^{n-i-2}S(n-i-2,k)\\
&=nB_{n-1}-B_n-(2n-5)B_{n-2}+\sum_{i=1}^{n-3}B_i
\end{align*}
non-initial occurrences of $a(a+2)$ in $\mathcal{P}_n$ altogether.

We now count the initial $a(a+2)$ occurrences in $\mathcal{P}_{n,k}$.  Note that such an occurrence may be obtained by choosing some $i \in [3,k]$ and then inserting $i-2$ directly prior to the initial $i$ within a member of $\mathcal{P}_{n-1,k}$, where $k \geq 3$.  Hence, there are $(k-2)S(n-1,k)$ initial occurrences of $a(a+2)$ in $\mathcal{P}_{n,k}$ for each $k \geq 2$.  Summing over $k$ gives
$$\sum_{k=2}^{n-1}(k-2)S(n-1,k)=B_n-B_{n-1}-1-2(B_{n-1}-1)=B_n-3B_{n-1}+1$$
initial occurrences of $a(a+2)$ in $\mathcal{P}_n$.  Combining this formula with the prior one then yields
\begin{align*}
&nB_{n-1}-B_n-(2n-5)B_{n-2}+\sum_{i=1}^{n-3}B_i+(B_n-3B_{n-1}+1)\\
&=(n-3)B_{n-1}-(2n-5)B_{n-2}+\sum_{i=0}^{n-3}B_i
\end{align*}
occurrences of $a(a+2)$ altogether in the members of $\mathcal{P}_n$, as desired.  \hfill \qed \medskip

{\bf Proof of recurrence \eqref{bnrec}:} \medskip

We may assume $n \geq 2$, the initial conditions being clear, and that $q$ is a positive integer.  We enumerate the set $\mathcal{A}_{n,k}'$  of colored members of $\mathcal{A}_{n,k}$ wherein in each word all letters corresponding to the second position in an occurrence of $a+s$ or $a-s$ for some $a$ receives one of $q$ colors, where $s+1 \leq k \leq 2s$.  Let $\pi=\pi_1\cdots \pi_{n-1} \in \mathcal{A}_{n-1,k}'$.  Observe that there are $k-1$ letters $i \in [k]$ such that $i - \pi_{n-1} \neq \pm s$ if $\pi_{n-1} \in [k]-[k-s+1,s]$, with all $i \in [k]$ satisfying $i - \pi_{n-1} \neq \pm s$ if $\pi_{n-1} \in [k-s+1,s]$.  Further, note that if the final letter of a member $\pi \in \mathcal{A}_{n-1,k}'$ belongs to $[k-s+1,s]$, then it may be deleted, as it cannot be the second entry in a string of the form $a(a\pm s)$, and hence there are $(2s-k)b_{n-2}$ possibilities for such $\pi$.  Thus, there are
$$(k-1)(b_{n-1}-(2s-k)b_{n-2})+k(2s-k)b_{n-2}=(k-1)b_{n-1}+(2s-k)b_{n-2}$$
members of $\mathcal{A}_{n,k}'$ altogether in which the final letter is not part of an occurrence of $a(a \pm s)$.  Further, for each $\pi \in \mathcal{A}_{n-1,k}'$ such that $\pi_{n-1} \in [k]-[k-s+1,s]$, of which there are, by subtraction, $b_{n-1}-(2s-k)b_{n-2}$ possibilities, one may append $\pi_{n-1}+s$ if $\pi_{n-1} \in [k-s]$ and $\pi_{n-1}-s$ if $\pi_{n-1} \in [s+1,k]$ to obtain a member of $\mathcal{A}_{n,k}'$ where the final letter is colored.  Note that $s+1 \leq k \leq 2s$ implies that this yields in a unique manner all members of $\mathcal{A}_{n,k}'$ for which this is the case.  Thus, we get $q(b_{n-1}-(2s-k)b_{n-2})$ such members of $\mathcal{A}_{n,k}'$ and combining with the first case above gives
\begin{align*}
b_n&=(k-1)b_{n-1}+(2s-k)b_{n-2}+q(b_{n-1}-(2s-k)b_{n-2})\\
&=(k-1+q)b_{n-1}+(1-q)(2s-k)b_{n-2},
\end{align*}
as desired. \hfill \qed

\end{document}